\documentclass[12pt]{article}
\usepackage{mathrsfs}
\usepackage{amsfonts,amsthm, amssymb}
\usepackage[tbtags]{amsmath}
\usepackage{latexsym, euscript, epic, eepic}
\usepackage{lineno}
\usepackage{cite}

\newtheorem{defn}{Definition}[section]

\newtheorem{thm}{Theorem}[section]
\newtheorem{lem}{Lemma}[section]

\numberwithin{equation}{section}

\begin{document}

\title{The Conditional Variational Principle for Maps with the Pseudo-orbit Tracing Property
 \footnotetext {* Corresponding author}
  }
\author{Zheng Yin$^{1}$, Ercai Chen$^{*1,2}$\\
  \small   1 School of Mathematical Sciences and Institute of Mathematics, Nanjing Normal University,\\
   \small   Nanjing 210023, Jiangsu, P.R.China\\
    \small 2 Center of Nonlinear Science, Nanjing University,\\
     \small   Nanjing 210093, Jiangsu, P.R.China.\\
      \small    e-mail: zhengyinmail@126.com ecchen@njnu.edu.cn
}
\date{}
\maketitle

\begin{center}
 \begin{minipage}{120mm}
{\small {\bf Abstract.}
Let $(X,d,f)$ be a topological dynamical system, where $(X,d)$
is a compact metric space and $f:X\to X$ is a continuous map.
We define $n$-ordered empirical measure of $x\in X$ by
\begin{align*}
\mathscr{E}_n(x)=\frac{1}{n}\sum\limits_{i=0}^{n-1}\delta_{f^ix},
\end{align*}
where $\delta_y$ is the Dirac mass at $y$. Denote by $V(x)$ the set
of limit measures of the sequence of measures $\mathscr{E}_n(x).$
In this paper, we obtain conditional variational principles
for the topological entropy of
\begin{align*}
\Delta_{sub}(I)=\left\{x\in X:V(x)\subset I\right\},
\end{align*}
and
\begin{align*}
\Delta_{cap}(I)=\left\{x\in X:V(x)\cap I\neq\emptyset \right\}.
\end{align*}
in a transitive dynamical system with the pseudo-orbit tracing property, where
$I$ is a certain subset of $\mathscr M_{\rm inv}(X,f)$.}
\end{minipage}
 \end{center}

\vskip0.5cm
\section{Introduction}
In this paper, $(X,d,f)$ (or $(X, f)$ for short) is a topological dynamical system
means that $(X,d)$ is a compact metric space together with a continuous self-map $f:X\to X$.
Let $\mathscr{M}(X)$, $\mathscr M_{\rm inv}(X,f)$ and $\mathscr M_{\rm erg}(X,f)$ be
the set of all Borel probability measures, $f$-invariant probability measures and ergodic measures
respectively. It is well known that
both $\mathscr M(X,f)$ and $\mathscr M_{\rm inv}(X,f)$
equiped with weak* topology
are compact metrizable spaces.

The well-known notions of topological and measure
entropy constitute essential invariants in the characterization of the complexity of
a dynamical system in classical ergodic theory. The relationship between these two
quantities is the so-called variational principle. We focus on the dimension-like entropy,
which was introduced by Bowen \cite{Bow} for any subset in a topological dynamical system,
and give a conditional variational principle for certain noncompact subsets.
We define $n$-ordered empirical measure of $x\in X$ by
\begin{align*}
\mathscr{E}_n(x)=\frac{1}{n}\sum\limits_{i=0}^{n-1}\delta_{f^ix},
\end{align*}
where $\delta_y$ is the Dirac mass at $y$. Denote by $V(x)$ the set
of limit measures of the sequence of measures $\mathscr{E}_n(x).$
Let $G_\mu:=\left\{x\in X: V(x)=\{\mu\}\right\}$ be the
set of  generic points of $\mu.$
In 1973, Bowen \cite{Bow} proved that the measure-theoretical entropy of $\mu$ is
equal to the topological entropy of $G_\mu$, i.e. $h^B_{\rm top}
(G_\mu,f)=h_\mu(f)$ when $\mu$ is ergodic.
In 2007, Pfister and Sullivan \cite{PfiSul} showed that for any $f$-invariant measure
$\mu,$
\begin{align} \label{equ1.1}
 h^B_{\rm top} (G_\mu,f)=h_\mu(f),
\end{align}
when the topological dynamical system $(X,d,f)$ is endowed with
$g$-almost product property (a weaker form of specification).
They also showed that formula (\ref{equ1.1}) implies
\begin{equation}\begin{split}\label{equ1.2}
 & h^B_{\rm top} \left(\left\{x\in X:\lim\limits_{n\to\infty}\frac{1}{n}\sum\limits_{i=0}^{n-1}\varphi
 (f^ix) = \alpha\right\},f\right)  \\= & \sup\left\{h_\mu(f):\mu\in \mathscr M_{\rm inv}(X,f),\int \varphi d\mu=\alpha\right\}.
\end{split}\end{equation}
where $\varphi:X\to\mathbb{R}$ is a continuous function and $\alpha\in\mathbb{R}.$
In 2003, Takens and Verbitskiy \cite{TakVer} proved  the equality (\ref{equ1.2}) in the case of dynamical systems with specification property.

In this paper, we follow Olsen's framework \cite{Ols1,Ols2,Ols3,Ols4,OlsWin}.
Let
\begin{align*}
&\Delta_{sub}(I):=\left\{x\in X:V(x)\subset I\right\},\\
  &\Delta_{equ}(I):=\left\{x\in X: V(x)= I \right\},
\end{align*}
where $I$ is a certain subset of $\mathscr M_{\rm inv}(X,f)$.
There are many papers describing the size of $\Delta_{sub}(I)$ and $\Delta_{equ}(I)$.
The quantities include Hausdorff dimension, packing dimension, topological entropy and topological
pressure. Olsen studied the Hausdorff dimension and packing dimension of $\Delta_{sub}(I)$ and $\Delta_{equ}(I)$
in symbolic spaces. Pfister and Sullivan  \cite{PfiSul} obtained variational principle for the topological
entropy of $\Delta_{equ}(I)$ in dynamical systems satisfying $g$-almost product property
and the uniform separation property. Based on the work of  \cite{PfiSul}, Zhou and Chen \cite{ZhoChe} investigated the topological pressure of
$\Delta_{sub}(I)$ and $\Delta_{equ}(I)$.
We focus on the systems satisfying the pseudo-orbit tracing property.
The pseudo-orbit tracing property, also knows as the shadowing property, was 
introduced by Anosov and Bowen on hyperbolic aspect of
differentiable dynamics. The property led to fruitful results in the study of 
ergodic theory and qualitative theory of dynamical systems (see \cite{AokHir} and \cite{DenGriSig}).
Recently, Dong, Oprocha and Xueting Tian \cite{DonOprTia} showed that the set
of points with divergent Birkhoff averages is either empty or 
carries full topological entropy.  Let
\begin{align*}
\Delta_{cap}(I):=\left\{x\in X:V(x)\cap I\neq\emptyset \right\}.
\end{align*}
In this paper, we obtain conditional variational principles for
the topological entropy of $\Delta_{sub}(I)$ and $\Delta_{cap}(I)$
in a transitive dynamical system with the pseudo-orbit tracing property.
Our results can be applied to the study of multifractal structure of Birkhoff averages.


\section{Definitions and Main Results}
In this section, we present some definitions and state the main results.

Throughout this paper, we shall
use $\mathbb{N}$ to denote the set of positive integers.
Let $C(X)$ be the space of continuous
functions from $X$ to $\mathbb{R}$ with the sup norm.
For $\varphi\in C(X)$
and $n\in\mathbb{N}$ we denote $\sum_{i=0}^{n-1}\varphi(f^ix)$ by $S_n\varphi(x)$.
If $n\in \mathbb{N}$, we define a metric $d_n$ on $X$ by
\begin{align*}
d_n(x,y):=\max_{0\leq i\leq n-1}d(f^ix,f^iy).
\end{align*}
For every $\epsilon>0$, $n\in \mathbb{N}$ and a point $x\in X$, define
$B_n(x,\epsilon):=\{y\in X:d_n(x,y)<\epsilon\}$.

For $\mu,\nu\in \mathscr{M}(X)$, let $D$ be a metric compatible with the weak* convergence on
$\mathscr{M}(X)$ defined as follows:
\begin{align*}
D(\mu,\nu):=\sum\limits_{i\geq1}\frac{|\int\varphi_id\mu-\int\varphi_id\nu|}{2^{i+1}\|\varphi_i\|},
\end{align*}
where $\{\varphi_i\}_{i=1}^\infty$ is the dense subset of
$C(X)$. It is obvious that $D(\mu,\nu)\leq1$ for any $\mu,\nu\in \mathscr{M}(X).$
We use an equivalent metric on $X$, denoted by $d$,
\begin{align*}
d(x,y):=D(\delta_x,\delta_y).
\end{align*}
For $\nu\in\mathscr{M}(X)$ and $\epsilon>0$, define
\begin{align*}
B(\nu,\epsilon):=\{\mu\in \mathscr{M}(X):D(\nu,\mu)<\epsilon\},\overline{B}(\nu,\epsilon):=\{\mu\in \mathscr{M}(X):D(\nu,\mu)\leq\epsilon\}.
\end{align*}
For a subset $Z\subset X,$ let
$\mathscr M_{\rm inv}(Z,f)$ denote the subset of $\mathscr M_{\rm inv}(X,f)$
for which the measures $\mu$ satisfy $\mu(Z)=1$ and $\mathscr M_{\rm erg}(Z,f)$
denote those which are ergodic.

A TDS $(X,d,f)$ is topologically transitive if given nonempty open subsets $U,V$ of $X$, there
exists $n\in \mathbb{N}$ with $f^n(U)\cap V\neq\emptyset$.

\begin{defn}
Let $(X,d,f)$ be a TDS. Given a number $\delta>0$, a $\delta$-pseudo-orbit
is a sequence $\{x_i\}_{i\in\mathbb{N}}$ such that $d(f(x_i),x_{i+1})\leq\delta$
for every $i\geq0$.
\end{defn}

\begin{defn}
Let $(X,d,f)$ be a TDS. We say that $f$ has pseudo-orbit tracing property if for every
$\epsilon>0$, there is a $\delta>0$ such that every $\delta$-pseudo-orbit can be
$\epsilon$-shadowed by an actual orbit, i.e. if $\{x_i\}_{i\in\mathbb{N}}$ satisfies
$d(f(x_i),x_{i+1})\leq\delta$ for every $i\geq0$, then there is a $x\in X$ such that
$d(f^i(x),x_{i})<\epsilon$ for all $i\geq0$.
\end{defn}

We give the definition of
topological entropy, which was introduced by Bowen introduced in \cite{Bow}.
This study defines it in an alternative way for convenience (see \cite{Pes}).

\begin{defn}(Bowen's topological entropy)
Given $Z\subset X,\epsilon>0$ and $N\in \mathbb{N},$ let
$\Gamma_N(Z,\epsilon)$ be the collection of all finite or countable
covers of $Z$ by sets of the form $B_n(x,\epsilon)$ with $n\geq N$.
For each $s\in\mathbb{R},$ we set
\begin{align*}
m(Z,s,N,\epsilon)=\inf\bigg\{\sum\limits_{B_n(x,\epsilon)\in\mathcal{C}}e^{-ns}:\mathcal{C}\in \Gamma_N(Z,\epsilon)\bigg\},
\end{align*}
and
\begin{align*}
m(Z,s,\epsilon)=\lim_{N\to\infty}m(Z,s,N,\epsilon).
\end{align*}
Define
\begin{align*}
h_{\rm top}(Z,\epsilon)&=\inf\{s\in\mathbb{R}:m(Z,s,\epsilon)=0\}=\sup\{s\in\mathbb{R}:m(Z,s,\epsilon)=\infty\},
\end{align*}
and topological entropy of $Z$ is
\begin{align*}
h^B_{\rm top}(Z):=\lim_{\epsilon\rightarrow0}h^B_{\rm top}(Z,\epsilon).
\end{align*}
\end{defn}

It is  obvious that the following hold:
\begin{enumerate}
  \item[(1)] $h^B_{\rm top}(Z_1)\leq h^B_{\rm top}(Z_2)$ for any $Z_1\subset Z_2\subset X$;
  \item[(2)] $h^B_{\rm top}(Z)=\sup_i h^B_{\rm top}(Z_i)$, where $Z=\bigcup_{i=1}^\infty Z_i\subset X$.
\end{enumerate}

Now we state our main results as follows.
\begin{thm}(Conditional Variational Principle)\label{thm2.1}
Let $(X,f)$ be a transitive TDS satisfying the
pseudo-orbit tracing property.
If $U$ is an nonempty open subset of $\mathscr M_{\rm inv}(X,f)$, then
\begin{align*}
h^B_{top}(\Delta_{sub}(U),f)=h^B_{top}(\Delta_{cap}(U),f)=\sup\limits_{\nu\in U}h_\nu(f).
\end{align*}
\end{thm}

\begin{thm}(Conditional Variational Principle)\label{thm2.2}
Let $(X,f)$ be a transitive TDS satisfying the
pseudo-orbit tracing property.
If $K\subset \mathscr M_{\rm inv}(X,f)$ is a convex subset and $intK\neq\emptyset$ in $\mathscr M_{\rm inv}(X,f)$, then
\begin{align*}
h^B_{top}(\Delta_{sub}(K),f)=h^B_{top}(\Delta_{cap}(K),f)=\sup\limits_{\nu\in K}h_\nu(f).
\end{align*}
\end{thm}

\begin{thm}(Shrinking Variational Principle)\label{thm2.3}
Let $(X,f)$ be a transitive TDS satisfying the
pseudo-orbit tracing property. If $\nu\in \mathscr M_{\rm inv}(X,f)$ and the map $\mu\mapsto h_\mu(f)$ is upper semi-continuous, then
\begin{align*}
\lim_{\delta\to0}h^B_{top}(\Delta_{sub}(B(\nu,\delta)),f)=h_\nu(f).
\end{align*}
\end{thm}

The theorem \ref{thm2.2} and theorem \ref{thm2.3} are motivated by the work of Mijovi\'{c} and Olsen \cite{MijOls}.
If $\mu\in\mathscr M_{\rm erg}(X,f)$ is an ergodic measure, then the subsystem $({\rm supp}(\mu),f)$
is topologically transitive. If $(X,f)$ is not topologically transitive, we can obtain the following result.

\begin{thm}\label{thm2.4}
Let $(X,f)$ be a TDS satisfying the
pseudo-orbit tracing property and $\mu$ be an ergodic measure.
We have the following results.
\begin{enumerate}
\item If $U\subset \mathscr M_{\rm inv}({\rm supp}(\mu),f)$ is an nonempty open subset of $\mathscr M_{\rm inv}(X,f)$, then
 \begin{align*}
   h^B_{top}(\Delta_{sub}(U),f)=h^B_{top}(\Delta_{cap}(U),f)=\sup\limits_{\nu\in U}h_\nu(f).
 \end{align*}
\item If $K\subset \mathscr M_{\rm inv}({\rm supp}(\mu),f)$ is a convex subset and $intK\neq\emptyset$ in $\mathscr M_{\rm inv}(X,f)$, then
 \begin{align*}
   h^B_{top}(\Delta_{sub}(K),f)=h^B_{top}(\Delta_{cap}(K),f)=\sup\limits_{\nu\in K}h_\nu(f).
 \end{align*}
\end{enumerate}
\end{thm}

\section{Examples}
We present two examples satisfying pseudo-orbit tracing property.

{\bf Example 1} If $f:[0,1]\to[0,1]$ is continuous and has fixed points only at the
ends of the interval, then $f$ has the pseudo-orbit tracing property (see lemma 4.1 in \cite{CheLi}).

{\bf Example 2 Tent maps}
Consider the family of tent maps, i.e., the piecewise linear maps $f_s:[0,2]\to[0,2]$, $1<s\leq2$ defined by
\begin{align*}
f_s(x)=
\left\{
  \begin{array}{ll}
    sx,     & 0\leq x\leq 1;\\
    s(2-x), & 1\leq x\leq 2.
  \end{array}
\right.
\end{align*}
Coven, Kan and Yorke \cite{CovKanYor} proved that the set of parameters $s$ for which $f_s$ has the
pseudo-orbit tracing property has full Lebesgue measure and is residual.

\section{Proof of Theorem \ref{thm2.1}}
In this section, we will verify theorem \ref{thm2.1}.
The upper bound on $h^B_{top}(\Delta_{cap}(U),f)$ is easy to get.
To obtain the lower bound estimate we need to construct
a suitable pseudo-orbit. The proof will be
divided into the following two subsections.

\subsection{Upper Bound on $h^B_{top}(\Delta_{cap}(U),f)$}
To obtain the upper bound, we need the following lemma.

\begin{lem}\label{lem4.1}{\rm(\cite{Bow})}
For $t\geq0$, consider the set
\begin{align*}
B(t)=\{x\in X:\exists \nu\in V(x) \text{ satisfying } h_\nu(f)\leq t\}.
\end{align*}
Then $h^B_{\rm top}(B(t))\leq t$.
\end{lem}

Let
\begin{align*}
t=\sup\left\{h_\nu(f):\nu\in U\right\}.
\end{align*}
Then $\Delta_{sub}(U)\subset\Delta_{cap}(U)\subset B(t)$.
Thus we have
\begin{align*}
h^B_{top}(\Delta_{sub}(U),f)\leq h^B_{top}(\Delta_{cap}(U),f)\leq\sup\left\{h_\nu(f):\nu\in U\right\}.
\end{align*}

\subsection{Lower Bound on $h^B_{top}(\Delta_{sub}(U),f)$}
The aim of this section is to obtain the lower bound of $h^B_{top}(\Delta_{sub}(U),f)$.

\subsubsection{Katok's Definition of Metric Entropy}
We use the Katok's definition of metric entropy based on the following
lemma.

\begin{lem}\label{lem4.2}{\rm (\cite{Kat})}
Let $(X,d)$ be a compact metric space, $f:X\to X$ be a continuous map and $\nu$
be an ergodic invariant measure. For $\epsilon>0$, $\delta\in (0,1)$ let $N^\nu(n,\epsilon,\delta)$ denote the minimum
number of $\epsilon$-Bowen balls $B_n(x,\epsilon)$, which cover a set of $\nu$-measure larger than $1-\delta$. Then
\begin{align*}
h_\nu(f)=\lim_{\epsilon\to0}\liminf_{n\to\infty}\frac{1}{n}\log N^\nu(n,\epsilon,\delta)=
\lim_{\epsilon\to0}\limsup_{n\to\infty}\frac{1}{n}\log N^\nu(n,\epsilon,\delta).
\end{align*}
\end{lem}

Fix $\delta\in (0,1)$. For $\epsilon>0$ and $\nu\in\mathscr{M}_{erg}(M,f)$, we define
\begin{align*}
h_\nu^{Kat}(f,\epsilon):=\liminf_{n\to\infty}\frac{1}{n}\log N^\nu(n,\epsilon,\delta).
\end{align*}
Then by lemma \ref{lem4.2},
\begin{align*}
h_\nu(f)=\lim_{\epsilon\to0}h_\nu^{Kat}(f,\epsilon).
\end{align*}
If $\nu$ is non-ergodic, we will define $h_\nu^{Kat}(f,\epsilon)$ by the ergodic
decomposition of $\nu$. The following lemma is necessary.

\begin{lem}\label{lem4.3}
Fix $\epsilon,\delta>0$ and $n\in\mathbb{N}$, the function $s:\mathscr{M}_{erg}(X,f)\to\mathbb{R}$
defined by $\nu\mapsto N^\nu(n,\epsilon,\delta)$ is upper semi-continuous.
\end{lem}
\begin{proof}
Let $\nu_k\to\nu$. Let $a>N^\nu(n,\epsilon,\delta)$; then there exists a set $S$
which $(n,\epsilon)$ spans some set $Z$ with $\nu(Z)>1-\delta$ such that
$a>\#S$, where $\#S$ denotes the number of elements in $S$.
If $k$ is large enough, then $\nu_k(\bigcup_{x\in S}B_n(x,\epsilon))>1-\delta$, which implies that
\begin{align*}
a>N^{\nu_{k}}(n,\epsilon,\delta).
\end{align*}
Thus we obtain
\begin{align*}
N^\nu(n,\epsilon,\delta)\geq\limsup_{k\to\infty}N^{\nu_{k}}(n,\epsilon,\delta),
\end{align*}
which completes the proof.
\end{proof}

Lemma \ref{lem4.3} tells us that the function $\overline{s}:\mathscr{M}_{erg}(X,f)\to\mathbb{R}$
defined by
\begin{align*}
\overline{s}(m)=h_m^{Kat}(f,\epsilon)
\end{align*}
is measurable.
Assume $\nu=\int_{\mathscr{M}_{erg}(M,f)}md\tau(m)$ is the ergodic decomposition of $\nu$.
Define
\begin{align*}
h_\nu^{Kat}(f,\epsilon):=\int_{\mathscr{M}_{erg}(X,f)}h_m^{Kat}(f,\epsilon)d\tau(m).
\end{align*}
By monotone convergence theorem, we have
\begin{align}\label{equ3.1}
h_\nu(f)=\int_{\mathscr{M}_{erg}(X,f)}\lim_{\epsilon\to0}h_m^{Kat}(f,\epsilon)d\tau(m)
=\lim_{\epsilon\to0}h_\nu^{Kat}(f,\epsilon).
\end{align}

\subsubsection{Proof of Theorem \ref{thm2.1}}

In this section we prove theorem \ref{thm2.1}. Let
\begin{align*}
\mathbf{C}:=\sup\left\{h_\nu(f):\nu\in U\right\}.
\end{align*}
We may assume that $\mathbf{C}$ is finite and $\mathbf{C}>0$. The case that $\mathbf{C}$ is infinite can be included in our proof.
Fix small $0<\delta,\gamma<1$ and $\gamma<\mathbf{C}/5$.
Choose a $\nu\in U$ such that
\begin{align*}
h_{\nu}(f)>\mathbf{C}-\gamma/2.
\end{align*}
By (\ref{equ3.1}), we can choose $\epsilon>0$ sufficiently small so that
\begin{align*}
\overline{B}(\nu,\frac{\epsilon}{2})\cap \mathscr{M}_{\rm inv}(X,f)\subset U,
h^{Kat}_{\nu}(f,\epsilon)>\mathbf{C}-\gamma.
\end{align*}
Then
\begin{align*}
h^{Kat}_{\nu}(f,\epsilon)-4\gamma>\mathbf{C}-5\gamma>0.
\end{align*}

\begin{lem}\label{lem4.4}
For any integer $k\geq1$,
there exists a finite convex combination of ergodic probability measures with
rational coefficients $\nu_k=\sum\limits_{j=1}^{s_k}a_{k,j}m_{k,j}$ such that
\begin{align*}
D(\nu,\nu_k)\leq\frac{1}{k}~{\rm and} ~
h_{\nu}^{Kat}(f,\epsilon)\leq\sum_{j=1}^{s_{k}}a_{k,j}h_{m_{k,j}}^{Kat}(f,\epsilon).
\end{align*}
\end{lem}
\begin{proof}
Let
\begin{align*}
\nu=\int_{\mathscr{M}_{erg}(X,f)}md\tau(m)
\end{align*}
be the ergodic decomposition of $\nu$. Choose $N$ large enough such that
\begin{align*}
\sum\limits_{n=N+1}^{\infty}\frac{2}{2^{n+1}}<\frac{1}{3k}.
\end{align*}
We choose $\zeta>0$ such that $D(\nu_1,\nu_2)<\zeta$ implies that
\begin{align*}
\left|\int \varphi_n d\nu_1-\int \varphi_n d\nu_2\right|<\frac{\|\varphi_n\|}{3k},n=1,2,\cdots,N.
\end{align*}
Let $\{A_{k,1},A_{k,2},\cdots,A_{k,s_k}\}$ be a partition of $\mathscr{M}_{erg}(X,f)$
with diameter smaller than $\zeta$. For any $A_{k,j}$ there exists an ergodic $m_{k,j}\in A_{k,j}$ such that
\begin{align*}
\int_{A_{k,j}}h_m^{Kat}(f,\epsilon)d\tau(m)
\leq\tau(A_{k,j})h_{m_{k,j}}^{Kat}(f,\epsilon).
\end{align*}
Obviously, $h_{\nu}^{Kat}(f,\epsilon)\leq\sum_{j=1}^{s_{k}}\tau(A_{k,j})h_{m_{k,j}}^{Kat}(f,\epsilon)$.
Let us choose rational numbers $a_{k,j}>0$
such that
$$|a_{k,j}-\tau(A_{k,j})|<\frac{1}{3ks_k}$$
and
$$h_{\nu}^{Kat}(f,\epsilon)\leq\sum_{j=1}^{s_{k}}a_{k,j}h_{m_{k,j}}^{Kat}(f,\epsilon).$$
Let
\begin{align*}
\mu_k=\sum\limits_{j=1}^{s_k}a_{k,j}m_{k,j}.
\end{align*}
By ergodic decomposition theorem, one can readily verify that
\begin{align*}
\left|\int\varphi_n d\nu-\int \varphi_n d\mu_k\right|\leq\frac{2\|\varphi_n\|}{3k},n=1,\cdots,N.
\end{align*}
Thus, we obtain
\begin{align*}
D(\nu,\mu_k)\leq\frac{1}{k}.
\end{align*}
\end{proof}

For $\epsilon>0$, we choose a positive real number $\delta'$
such that every $\delta'$-pseudo-orbit can be
$\frac{\epsilon}{4}$-shadowed by an actual orbit.
Let $\xi$ be a finite partition of $X$ with
diam$(\xi)<\frac{\delta'}{3}$.
For $n\in\mathbb{N},$ we consider the set
\begin{align*}
\Lambda^n(m_{k,j})=\{x\in X:f^q(x)\in\xi(x){\rm ~for~some~} q\in[n,(1+\gamma)n]\cap\mathbb{N},\\
{\rm~and~}D(\mathscr{E}_m(x),m_{k,j})<\frac{1}{k} {\rm~for~all~} m\geq n\},
\end{align*}
where $\xi(x)$ is the element in $\xi$ containing $x.$

\begin{lem}\label{lem4.5}
\begin{align*}
\lim_{n\to\infty}m_{k,j}(\Lambda^n(m_{k,j}))=1.
\end{align*}
\end{lem}
\begin{proof}
Let $A$ be an element in $\xi$ with $m_{k,j}(A)>0$.
We show that there exists a measurable function $N_0:A\to\mathbb{N}$ such that for a.e. $x\in A$,
every $n\geq N_0(x)$ there is a $q\in [n,(1+\gamma)n]\cap\mathbb{N}$ such that $f^q(x)\in A$.
Let
\begin{align*}
s_n(x)=\sum_{i=0}^{n-1}\chi_{A}(f^i(x)).
\end{align*}
By Birkhoff ergodic theorem, we have
\begin{align*}
\lim_{n\to\infty}\frac{1}{n}s_n(x)=m_{k,j}(A)>0
\end{align*}
for a.e. $x\in A$.
Let $a=m_{k,j}(A)$.
Choose $0<\eta<a$ such that $(a+\eta)/(a-\eta)<1+\gamma/2$.
Let $n_0:A\to \mathbb{N}$ be a measurable function such that
for a.e. $x\in A$,
$n\geq n_0(x)$ implies $|(1/n)s_n(x)-a|<\eta$. Let
\begin{align*}
N_0(x)=\max\left\{n_0(x),\frac{2}{\gamma}\right\}.
\end{align*}
Suppose that there exists $n\geq N_0(x)$ such that $f^q(x)\notin A$ for every
$q\in [n,(1+\gamma)n]\cap\mathbb{N}$. Let $m$ be the largest integer in
$[n,(1+\gamma)n]\cap\mathbb{N}$. Then $m-n>(1+\gamma)n-1-n\geq n\gamma/2$.
We have
\begin{align*}
a-\eta<\frac{s_m(x)}{m}=\frac{s_n(x)}{m}<\frac{s_n(x)}{n(1+\gamma/2)}<\frac{a+\eta}{1+\gamma/2}<a-\eta,
\end{align*}
a contradiction. The desired result follows.
\end{proof}

By lemma \ref{lem4.5}, we can take $n_k\to\infty$ such that
\begin{align*}
m_{k,j}(\Lambda^n(m_{k,j}))>1-\delta
\end{align*}
for all $n\geq n_k$ and $1\leq j\leq s_k.$

For $k\in \mathbb{N}$, let
\begin{align*}
Q(\Lambda^n(m_{k,j}),\epsilon)&=\inf\{\sharp S:S \text{ is an } (n,\epsilon) \text{ spanning set for } \Lambda^n(m_{k,j}) \},\\
P(\Lambda^n(m_{k,j}),\epsilon)&=\sup\{\sharp S:S \text{ is an } (n,\epsilon) \text{ separated set for } \Lambda^n(m_{k,j}) \}.
\end{align*}
Then for all $n\geq n_k$ and $1\leq j\leq s_k$, we have
\begin{align*}
P(\Lambda^n(m_{k,j}),\epsilon)\geq Q(\Lambda^n(m_{k,j}),\epsilon)\geq N^{m_{k,j}}(n,\epsilon,\delta).
\end{align*}
We obtain
\begin{align*}
\liminf_{n\to\infty}\frac{1}{n}\log P(\Lambda^n(m_{k,j}),\epsilon)
\geq h^{Kat}_{m_{k,j}}(f,\epsilon).
\end{align*}
Thus for each $k\in \mathbb{N}$, we can choose $t_k$ large enough such that $\exp(\gamma t_k)>\sharp \xi$
and
\begin{align*}
\frac{1}{t_k}\log P(\Lambda^{t_k}(m_{k,j}),\epsilon)>h^{Kat}_{m_{k,j}}(f,\epsilon)-\gamma
\end{align*}
for $1\leq j\leq s_k$. Let $S(k,j)$
be a $(t_k,\epsilon)$-separated set for $\Lambda^{t_k}(m_{k,j})$ and
\begin{align*}
\# S(k,j)\geq\exp\left(t_k(h^{Kat}_{m_{k,j}}(f,\epsilon)-2\gamma)\right).
\end{align*}
For each $q\in[t_k,(1+\gamma)t_k]\cap\mathbb{N}$, let
\begin{align*}
V_q=\{x\in S(k,j):f^q(x)\in\xi(x)\}
\end{align*}
and let $n(k,j)$ be the value of $q$ which maximizes $\#V_q.$ Obviously,
$n(k,j)\geq t_k$ and $t_k\geq\frac{n(k,j)}{1+\gamma}\geq n(k,j)(1-\gamma).$
Since $\exp({\gamma t_k})\geq\gamma t_k+1,$ we have that
\begin{align*}
\#V_{n(k,j)}\geq\frac{\#S(k,j)}{\gamma t_k+1}\geq \exp\left(t_k(h^{Kat}_{m_{k,j}}(f,\epsilon)-3\gamma)\right).
\end{align*}
Consider the element $A_{n(k,j)}\in\xi$ such that $\#(V_{n(k,j)}\cap A_{n(k,j)})$ is maximal.
Let $W_{n(k,j)}=V_{n(k,j)}\cap A_{n(k,j)}$. It follows that
\begin{align*}
\#W_{n(k,j)}\geq \frac{1}{\#\xi}\#V_{n(k,j)}\geq \frac{1}{\#\xi}\exp\left(t_k(h^{Kat}_{m_{k,j}}(f,\epsilon)-3\gamma)\right).
\end{align*}
Since $\exp(\gamma t_k)>\sharp \xi$, $t_k\geq n(k,j)(1-\gamma)$, we have
\begin{align*}
\#W_{n(k,j)}\geq\exp\left(n(k,j)(1-\gamma)(h^{Kat}_{m_{k,j}}(f,\epsilon)-4\gamma)\right).
\end{align*}

Notice that $A_{n(k,j)}$ is contained in an open subset $U(k,j)$
with diam$(U(k,j))\leq3\text{diam}(\xi)<\delta'$. Since $(X,f)$ is transitive, for any two
measures $m_{k_1,j_1},m_{k_2,j_2}$, there exist $s=s(k_1,j_1,k_2,j_2)\in\mathbb{N}$
and $y=y(k_1,j_1,k_2,j_2)\in U(k_1,j_1)$  such that
$f^s(y)\in U(k_2,j_2)$.
Let $C_{k,j}=\frac{a_{k,j}}{n(k,j)}$. We can choose an integer $N_k$
large enough so that $N_kC_{k,j}$ are integers and
\begin{align}\label{equ1}
N_k\geq k\sum\limits_{\substack{1\leq r_1,r_2\leq k+1 \\ 1\leq j_i\leq s_{r_i},i=1,2}}s(r_1,j_1,r_2,j_2).
\end{align}
Let $X_k=\sum\limits_{j=1}^{s_k-1}s(k,j, k,j+1)+s(k,s_k, k,1)$  and
\begin{align}\label{equ2}
Y_k=\sum\limits_{j=1}^{s_k}N_kn(k,j)C_{k,j}+X_k=N_k+X_k,
\end{align}
then we have
\begin{align}
\frac{N_k}{Y_k}\geq \frac{1}{1+\frac{1}{k}}\geq 1-\frac{1}{k}.\label{equ3}
\end{align}
Choose a strictly increasing sequence $\{T_k\}$ with $T_k\in\mathbb{N},$
\begin{align}\label{equ4}
Y_{k+1}\leq\frac{1}{k+1}\sum\limits_{r=1}^kY_rT_r,
\sum\limits_{r=1}^k(Y_rT_r+s(r,1,r+1,1))\leq \frac{1}{k+1}Y_{k+1}T_{k+1}.
\end{align}
For $x\in X,$ we define segments of orbits
\begin{align*}
O_{k,j}(x)&:=(x,f(x),\cdots,f^{n(k,j)-1}(x)), 1\leq j\leq s_k,\\
\widehat{O}_{k_1,j_1,k_2,j_2}(x)&:=(x,f(x),\cdots,f^{s(k_1,j_1,k_2,j_2)-1}(x)),1\leq j_i\leq s_{k_i},i=1,2.
\end{align*}
Consider the $\delta$-pseudo-orbit with finite length
\begin{align*}
O_k=O( &x(1,1,1,1), \cdots,x(1,1,1,N_1C_{1,1}),\cdots,
x(1,s_1,1,1),\cdots,x(1,s_1,1,N_1C_{1,s_1});\\&\cdots;\\
&x(1,1,T_1,1),\cdots,x(1,1,T_1,N_1C_{1,1}),\cdots,x(1,s_1,T_1,1),\cdots, x(1,s_1,T_1,N_1C_{1,s_1});\\
&\vdots\\&
x(k,1,1,1), \cdots,x(k,1,1,N_kC_{k,1}),\cdots,
x(k,s_k,1,1),\cdots,x(k,s_k,1,N_kC_{k,s_k});\\&\cdots;\\&
x(k,1,T_k,1), \cdots,x(k,1,T_k,N_kC_{k,1}),\cdots,
x(k,s_k,T_k,1),\cdots,x(k,s_k,T_k,N_kC_{k,s_k});
)
\end{align*}
with the precise form as follows:
\begin{align*}
\{ &O_{1,1}(x(1,1,1,1)), \cdots,O_{1,1}(x(1,1,1,N_1C_{1,1})),\widehat{O}_{1,1,1,2}(y(1,1,1,2));\\&
O_{1,2}(x(1,2,1,1)), \cdots,O_{1,2}(x(1,2,1,N_1C_{1,2})),\widehat{O}_{1,2,1,3}(y(1,2,1,3));
\cdots,\\&
O_{1,s_1}(x(1,s_1,1,1)), \cdots,O_{1,s_1}(x(1,s_1,1,N_1C_{1,s_1})),\widehat{O}_{1,s_1,1,1}(y(1,s_1,1,1));\\&
\cdots,\\&
O_{1,1}(x(1,1,T_1,1)), \cdots,O_{1,1}(x(1,1,T_1,N_1C_{1,1})),\widehat{O}_{1,1,1,2}(y(1,1,1,2));\\&
O_{1,2}(x(1,2,T_1,1)), \cdots,O_{1,2}(x(1,2,T_1,N_1C_{1,2})),\widehat{O}_{1,2,1,3}(y(1,2,1,3));
\cdots,\\&
O_{1,s_1}(x(1,s_1,T_1,1)), \cdots,O_{1,s_1}(x(1,s_1,T_1,N_1C_{1,s_1})),\widehat{O}_{1,s_1,1,1}(y(1,s_1,1,1));\\&
\widehat{O}(y(1,1,2,1));\\&\vdots,\\&
O_{k,1}(x(k,1,1,1)), \cdots,O_{k,1}(x(k,1,1,N_kC_{k,1})),\widehat{O}_{k,1,k,2}(y(k,1,k,2));\\&
O_{k,2}(x(k,2,1,1)), \cdots,O_{k,2}(x(k,2,1,N_kC_{k,2})),\widehat{O}_{k,2,k,3}(y(k,2,k,3));\cdots\\&
O_{k,s_k}(x(k,s_k,1,1)), \cdots,O_{k,s_k}(x(k,s_k,1,N_kC_{k,s_k})),\widehat{O}_{k,s_k,k,1}(y(k,s_k,k,1));\\
&\cdots\\&
O_{k,1}(x(k,1,T_k,1)), \cdots,O_{k,1}(x(k,1,T_k,N_kC_{k,1})),\widehat{O}_{k,1,k,2}(y(k,1,k,2));\\&
O_{k,2}(x(k,2,T_k,1)), \cdots,O_{k,2}(x(k,2,T_k,N_kC_{k,2})),\widehat{O}_{k,2,k,3}(y(k,2,k,3));\cdots\\&
O_{k,s_k}(x(k,s_k,T_k,1)), \cdots,O_{k,s_k}(x(k,s_k,T_k,N_kC_{k,s_k})),\widehat{O}_{k,s_k,k,1}(y(k,s_k,k,1));\\&
\widehat{O}(y(k,1,k+1,1));
\},
\end{align*}
where $x(q,j,i,t)\in W_{n(q,j)}.$

For $1\leq q\leq k,1\leq i\leq T_q,1\leq j\leq s_q, 1\leq t\leq N_qC_{q,j},$ let $M_1=0,$
\begin{align*}
M_q&=M_{q,1}=\sum\limits_{r=1}^{q-1}(T_rY_r+s(r,1,r+1,1)),\\
M_{q,i}&=M_{q,i,1}=M_q+(i-1)Y_q,\\
M_{q,i,j}&=M_{q,i,j,1}=M_{q,i}+\sum\limits_{p=1}^{j-1}(N_qn(q,p)C_{q,p}+s(q,p,q,p+1)),
\\
M_{q,i,j,t}&=M_{q,i,j}+(t-1)n(q,j).
\end{align*}
By pseudo-orbit tracing property, there exists at least one shadowing point $z$ of $O_k$ such that
\begin{align*}
d(f^{M_{q,i,j,t}+p}(z),f^p(x(q,j,i,t)))\leq\frac{\epsilon}{4},
\end{align*}
for $1\leq q\leq k, 1\leq i\leq T_q, 1\leq j\leq s_q, 1\leq t\leq N_qC_{q,j}, 0\leq p\leq n(q,j)-1.$
Let $B(x(1,1,1,1),\cdots,x(k,s_k,T_k,N_kC_{k,s_k}))$
be the set of all shadowing points for the above pseudo-orbit.
Then the set $B(x(1,1,1,1),\cdots,x(k,s_k,T_k,N_kC_{k,s_k}))$ can be considered
as a map with variables $x(q,j,i,t)$.
We define $F_k$ by
\begin{align*}
F_k=\bigcup\{B(&x(1,1,1,1),\cdots,x(k,s_k,T_k,N_kC_{k,s_k})):\\
&x(1,1,1,1)\in W_{n(1,1)},\cdots,x(k,s_k,T_k,N_kC_{k,s_k})\in W_{n(k,s_k)}\}.
\end{align*}
Obviously, $F_k$ is non-empty compact and $F_{k+1}\subseteq F_{k}$. Define $F=\bigcap_{k=1}^{\infty}F_k$.

\begin{lem}
$F\subset \Delta_{sub}(U)$.
\end{lem}
\begin{proof}
It suffices to prove that for any $z\in F$,
\begin{align*}
\limsup_{n\to\infty}D\left(\mathscr{E}_n(z),\nu\right)\leq\frac{\epsilon}{2}.
\end{align*}
Assume that $z\in B(x(1,1,1,1),\cdots,x(k,s_k,T_k,N_kC_{k,s_k}))$.
We firstly show that
\begin{align*}
D\left(\mathscr{E}_{M_{k+1}}(z),\nu\right)\leq\frac{\epsilon}{2}
\end{align*}
when $k$ is large enough.
It is obvious that
\begin{align*}
     &D\left(\mathscr{E}_{T_kY_k+s(k,1,k+1,1)}(f^{M_k}(z)),\nu\right)\\
\leq &D\left(\mathscr{E}_{T_kY_k+s(k,1,k+1,1)}(f^{M_k}(z)),\mathscr{E}_{T_kY_k}(f^{M_k}(z))\right)\\
     &+D\left(\mathscr{E}_{T_kY_k}(f^{M_k}(z)),\frac{1}{T_kN_k}\sum_{i=1}^{T_k}\sum_{j=1}^{s_k}\sum_{t=1}^{N_kC_{k,j}}\sum_{q=0}^{n(k,j)-1}\delta_{f^{M_{k,i,j,t+q}}(z)}\right)\\
     &+D\left(\frac{1}{T_kN_k}\sum_{i=1}^{T_k}\sum_{j=1}^{s_k}\sum_{t=1}^{N_kC_{k,j}}\sum_{q=0}^{n(k,j)-1}\delta_{f^{M_{k,i,j,t+q}}(z)},\nu_k\right)+D(\nu_k,\nu),
\end{align*}
where $D(\nu_k,\nu)\leq\frac{1}{k}\to0$ as $k\to\infty$.
For any $\psi\in C(X)$, by (\ref{equ1}) we have
\begin{align*}
    &\left|\int\psi d\mathscr{E}_{T_kY_k+s(k,1,k+1,1)}\psi(f^{M_k}(z))-\int\psi d\mathscr{E}_{T_kY_k}\psi(f^{M_k}(z))\right|\\
   =&\left|\frac{S_{T_kY_k+s(k,1,k+1,1)}\psi(f^{M_k}(z))}{T_kY_k+s(k,1,k+1,1)}-\frac{S_{T_kY_k}\psi(f^{M_k}(z))}{T_kY_k}\right|\\
\leq&\left|\frac{S_{T_kY_k}\psi(f^{M_k}(z))}{T_kY_k+s(k,1,k+1,1)}\left(1-\frac{T_kY_k+s(k,1,k+1,1)}{T_kY_k}\right)\right|+\frac{s(k,1,k+1,1)\|\psi\|}{T_kY_k+s(k,1,k+1,1)}\\
\to &0
\end{align*}
as $k\to\infty$. It follows that
\begin{align*}
\lim_{k\to\infty}D(\mathscr{E}_{T_kY_k+s(k,1,k+1,1)}(f^{M_k}(z)),\mathscr{E}_{T_kY_k}(f^{M_k}(z)))=0.
\end{align*}
One can also prove that
\begin{align*}
\lim_{k\to\infty}D\left(\mathscr{E}_{T_kY_k}(f^{M_k}(z)),\frac{1}{T_kN_k}\sum_{i=1}^{T_k}\sum_{j=1}^{s_k}\sum_{t=1}^{N_kC_{k,j}}\sum_{q=0}^{n(k,j)-1}\delta_{f^{M_{k,i,j,t+q}}(z)}\right)=0
\end{align*}
in the same way.

Since $C_{k,j}n(k,j)=a_{k,j}$, we have
\begin{align*}
    &D\left(\frac{1}{T_kN_k}\sum_{i=1}^{T_k}\sum_{j=1}^{s_k}\sum_{t=1}^{N_kC_{k,j}}\sum_{q=0}^{n(k,j)-1}\delta_{f^{M_{k,i,j,t+q}}(z)},\nu_k\right) \\
\leq&D\left(\frac{1}{T_kN_k}\sum_{i=1}^{T_k}\sum_{j=1}^{s_k}\sum_{t=1}^{N_kC_{k,j}}\sum_{q=0}^{n(k,j)-1}\delta_{f^{M_{k,i,j,t+q}}(z)},
            \frac{1}{T_kN_k}\sum_{i=1}^{T_k}\sum_{j=1}^{s_k}\sum_{t=1}^{N_kC_{k,j}}\sum_{q=0}^{n(k,j)-1}\delta_{f^{M_{k,i,j,t+q}}(x(k,j,i,t))}\right)\\
    &+D\left(\frac{1}{T_kN_k}\sum_{i=1}^{T_k}\sum_{j=1}^{s_k}\sum_{t=1}^{N_kC_{k,j}}\sum_{q=0}^{n(k,j)-1}\delta_{f^{M_{k,i,j,t+q}}(x(k,j,i,t))},
       \frac{1}{T_kN_k}\sum_{j=1}^{s_k}T_kN_kC_{k,j}n(k,j)m_{k,j}\right)\\
\leq&\frac{\epsilon}{4}+D\left(\frac{1}{T_kN_k}\sum_{i=1}^{T_k}\sum_{j=1}^{s_k}\sum_{t=1}^{N_kC_{k,j}}n(k,j)\mathscr{E}_{n(k,j)}(x(k,j,i,t)),
       \frac{1}{T_kN_k}\sum_{i=1}^{T_k}\sum_{j=1}^{s_k}\sum_{t=1}^{N_kC_{k,j}}n(k,j)m_{k,j}\right)\\
\leq&\frac{\epsilon}{4}+\sum_{i=1}^{T_k}\sum_{j=1}^{s_k}\sum_{t=1}^{N_kC_{k,j}}\frac{n(k,j)}{T_kN_k}D\left(\mathscr{E}_{n(k,j)}(x(k,j,i,t),m_{k,j}\right)\\
\leq&\frac{\epsilon}{4}+\frac{1}{k}.
\end{align*}
It follows that
\begin{align*}
\limsup_{k\to\infty}D\left(\mathscr{E}_{T_kY_k+s(k,1,k+1,1)}(f^{M_k}(z)),\nu\right)\leq\frac{\epsilon}{4}.
\end{align*}
By inequalities (\ref{equ1}) and (\ref{equ4}),
one can readily verify that
\begin{align*}
\lim_{k\to\infty}\frac{T_kY_k+s(k,1,k+1,1)}{M_{k+1}}=1.
\end{align*}
For any $\psi\in C(X)$,
\begin{align*}
 &\left|\int\psi d\mathscr{E}_{M_{k+1}(z)}-\int\psi d\mathscr{E}_{T_kY_k+s(k,1,k+1,1)}(f^{M_k}(z))\right|\\
=&\left|\frac{1}{M_{k+1}}S_{M_k}\psi(z)
+\frac{S_{T_kY_k+s(k,1,k+1,1)}\psi(f^{M_k}(z))}{M_{k+1}}
-\frac{S_{T_kY_k+s(k,1,k+1,1)}\psi(f^{M_k}(z))}{T_kY_k+s(k,1,k+1,1)}\right|\\
=&\left|\frac{1}{M_{k+1}}S_{M_k}\psi(z)
+\frac{S_{T_kY_k+s(k,1,k+1,1)}\psi(f^{M_k}(z))}{T_kY_k+s(k,1,k+1,1)}
\left(\frac{T_kY_k+s(k,1,k+1,1)}{M_{k+1}}-1\right)\right|\\
\leq&\frac{M_k}{M_{k+1}}\|\psi\|+\|\psi\|\left|\frac{T_kY_k+s(k,1,k+1,1)}{M_{k+1}}-1\right|\to0
\end{align*}
as $k\to\infty$. We deduce that
\begin{align*}
\limsup_{k\to\infty}D\left(\mathscr{E}_{M_{k+1}}(z),\nu\right)\leq\frac{\epsilon}{4}.
\end{align*}
We consider $M_{k}\leq n< M_{k+1}$. There exists $1\leq m\leq T_k$ such that $M_{k,m}\leq n< M_{k,m+1}$.
Here we appoint $M_{k,T_k+1}=M_{k+1}$. We consider the case $1< m\leq T_k$.
The case $m=1$ is similar. It follows that
\begin{align*}
     &D\left(\mathscr{E}_n(z),\nu\right)\\
\leq &\frac{M_k}{n}D\left(\mathscr{E}_{M_k}(z),\nu\right)+\frac{1}{n}\sum_{i=1}^{m-2}D\left(\mathscr{E}_{Y_k}(f^{M_{k,i}}(z)),\nu\right)
          +\frac{n-M_{k,m}}{n}D\left(\mathscr{E}_{n-M_{k,m}}(f^{M_{k,m}}(z)),\nu\right)\\
\leq &\frac{M_k}{n}D\left(\mathscr{E}_{M_k}(z),\nu\right)
       +\frac{1}{n}\sum_{i=1}^{m-2}D\left(\mathscr{E}_{Y_k}(f^{M_{k,i}}(z)),\frac{1}{N_k}\sum_{j=1}^{s_k}\sum_{t=1}^{N_kC_{k,j}}\sum_{q=0}^{n(k,j)-1}\delta_{f^{M_{k,i,j,t+q}}(z)}\right)\\
     &+\frac{1}{n}\sum_{i=1}^{m-2}D\left(\frac{1}{N_k}\sum_{j=1}^{s_k}\sum_{t=1}^{N_kC_{k,j}}\sum_{q=0}^{n(k,j)-1}\delta_{f^{M_{k,i,j,t+q}}(z)},\nu_k\right)
      +D\left(\nu_k,\nu\right)+\frac{Y_k+s(k,1,k+1,1)}{n}.
\end{align*}
For $1\leq i\leq m-2$, we have
\begin{align*}
\lim_{k\to\infty}D\left(\mathscr{E}_{Y_k}(f^{M_{k,i}}(z)),\frac{1}{N_k}\sum_{j=1}^{s_k}\sum_{t=1}^{N_kC_{k,j}}\sum_{q=0}^{n(k,j)-1}\delta_{f^{M_{k,i,j,t+q}}(z)}\right)=0,
\end{align*}
and
\begin{align*}
\limsup_{k\to\infty}D\left(\frac{1}{N_k}\sum_{j=1}^{s_k}\sum_{t=1}^{N_kC_{k,j}}\sum_{q=0}^{n(k,j)-1}\delta_{f^{M_{k,i,j,t+q}}(z)},\nu_k\right)\leq\frac{\epsilon}{4}.
\end{align*}
Thus we have
\begin{align*}
\limsup_{n\to\infty}D\left(\mathscr{E}_n(z),\nu\right)\leq\frac{\epsilon}{2},
\end{align*}
which completes the proof.
\end{proof}

Next, we compute the topological entropy of $F$. Fix the position indexed
$m,j,i,t,$ for distinct $x(m,j,i,t),x'(m,j,i,t)\in W_{n(m,j)},$ the corresponding shadowing points $z,z'$ satisfying
\begin{equation}\label{equ4.6}
\begin{split}
 &d(f^{M_{m,i,j,t}+q}(z), f^{M_{m,i,j,t}+q}(z'))\\
 \geq &d(f^q(x(m,j,i,t)),f^q(x'(m,j,i,t)))-d(f^{M_{m,i,j,t}+q}(z),f^q(x(m,j,i,t)))\\
       &-d(f^{M_{m,i,j,t}+q}(z'),f^q(x'(m,j,i,t)))\\
 \geq &d(f^q(x(m,j,i,t)),f^q(x'(m,j,i,t)))-\frac{\epsilon}{2}.
\end{split}
\end{equation}
Noticing that $x(m,j,i,t),x'(m,j,i,t)$ are $(n(m,j),\epsilon)$-separated, we obtain
$f^{M_{m,i,j,t }}(z)$, $f^{M_{m,i,j,t }}(z')$ are $(n(m,j),\epsilon/2)$-separated.

Since $F$ is compact we can consider finite covers
$\mathcal{C}$ of $F$ with the property that if $B_m(x,\epsilon/2)\in\mathcal{C}$, then
$B_m(x,\epsilon/2)\cap F\neq\emptyset$. By definition
\begin{align*}
m(F,s,N,\epsilon/2)=\inf\left\{\sum\limits_{B_n(x,\epsilon/2)\in\mathcal{C}}e^{-ns}:\mathcal{C}\in \Gamma_N(F,\epsilon/2)\right\}.
\end{align*}
For each $\mathcal{C}\in \Gamma_N(Z,\epsilon/2)$ we define a new cover $\mathcal{C}'$ in which for $M_{k,i}\leq m<M_{k,i+1}$,
$B_m(x,\epsilon/2)$ is replaced by $B_{M_{k,i}}(x,\epsilon/2)$. Here we appoint $M_{k,T_k+1}=M_{k+1}$. Then
\begin{align*}
m(F,s,N,\epsilon/2)=\inf_{\mathcal{C}\in \Gamma_N(F,\epsilon/2)}\sum\limits_{B_n(x,\epsilon/2)\in\mathcal{C}}e^{-ns}
                  \geq \inf_{\mathcal{C}\in \Gamma_N(F,\epsilon/2)}\sum\limits_{B_{M_{k,i}}(x,\epsilon/2)\in\mathcal{C}'}e^{-M_{k,i+1}s}.
\end{align*}
We use the lexicographical order for the set $\{(k,i):k,i\in\mathbb{N},1\leq i\leq T_k\}$. Let
\begin{align*}
(k_0,i_0)=\max\left\{(k,i):B_{M_{k,i}}(x,\epsilon/2)\in\mathcal{C}'\right\}
\end{align*}
and
\begin{align*}
\mathcal{M}_{k,i}=\big\{&x(k,1,i,1),\cdots,x(k,1,i,N_kC_{k,1}),\cdots,x(k,j,i,1),\cdots,x(k,j,i,N_kC_{k,1}),\cdots,\\
                      &x(k,s_k,i,1),\cdots,x(k,s_k,i,N_kC_{k,s_k}):
                      x(k,j,i,t)\in W_{n(k,j)},1\leq j\leq s_k,\\
                                 &1\leq t\leq N_kC_{k,j}\big\},
\end{align*}
where $1\leq i\leq T_k$. Define
\begin{align*}
\mathcal{W}_{k,i}:=\prod_{(m,n)<(k,i)}\mathcal{M}_{m,n},
\overline{\mathcal{W}_{k_0,i_0}}:=\bigcup_{(k,i)\leq(k_0,i_0)}\mathcal{W}_{k,i}.
\end{align*}
By (\ref{equ4.6}), each $x\in B_{M_{k,i}}\cap F$ corresponds to a unique point in $\mathcal{W}_{k,i}$.
For $(k,i)\leq (k_0,i_0)$, each $w\in \mathcal{W}_{k,i}$ is the prefix of exactly
$\#\mathcal{W}_{k_0,i_0}/\#\mathcal{W}_{k,i}$ elements of $\mathcal{W}_{k_0,i_0}$.
If $\mathcal{W}\subset\overline{\mathcal{W}_{k_0,i_0}}$ contains a prefix of each element of
$\mathcal{W}_{k_0,i_0}$, then
\begin{align*}
\sum_{(k,i)\leq(k_0,i_0)}\#(\mathcal{W}\cap \mathcal{W}_{k,i})\frac{\#\mathcal{W}_{k_0,i_0}}{\#\mathcal{W}_{k,i}}\geq\#\mathcal{W}_{k_0,i_0},
\end{align*}
i.e.
\begin{align*}
\sum_{(k,i)\leq(k_0,i_0)}\frac{\#(\mathcal{W}\cap \mathcal{W}_{k,i})}{\#\mathcal{W}_{k,i}}\geq 1.
\end{align*}

It follows from
\begin{align*}
\#W_{n(k,j)}\geq\exp\left(n(k,j)(1-\gamma)(h^{Kat}_{m_{k,j}}(f,\epsilon)-4\gamma)\right)
\end{align*}
and lemma \ref{lem4.4} that
\begin{align*}
\#\mathcal{W}_{k,i}\geq&\left(\sharp W_{n(1,1)}^{N_1C_{1,1}}\cdots\sharp W_{n(1,s_1)}^{N_1C_{1,s_1}}\right)^{T_1}\cdots
                          \left(\sharp W_{n(k-1,1)}^{N_{k-1}C_{k-1,1}}\cdots\sharp W_{n(k-1,s_{k-1})}^{N_{k-1}C_{k-1,s_{k-1}}}\right)^{T_{k-1}}\\
                       &\cdots\left(\sharp W_{n(k,1)}^{N_{k}C_{k,1}}\cdots\sharp W_{n(k,s_{k})}^{N_{k}C_{k,s_{k}}}\right)^{i-1},
\end{align*}
where
\begin{align*}
    &\left(\sharp W_{n(1,1)}^{N_1C_{1,1}}\cdots\sharp W_{n(1,s_1)}^{N_1C_{1,s_1}}\right)^{T_1}\cdots
                          \left(\sharp W_{n(k-1,1)}^{N_{k-1}C_{k-1,1}}\cdots\sharp W_{n(k-1,s_{k-1})}^{N_{k-1}C_{k-1,s_{k-1}}}\right)^{T_{k-1}}\\
\geq&\exp\left(\sum_{l=1}^{k-1}\sum_{j=1}^{s_l}T_lN_lC_{l,j}n(l,j)(1-\gamma)(h^{Kat}_{m_{l,j}}(f,\epsilon)-4\gamma)\right)\\
\geq&\exp\left(\sum_{l=1}^{k-1}T_lN_l(1-\gamma)(h^{Kat}_{\nu}(f,\epsilon)-4\gamma)\right)
\end{align*}
and
\begin{align*}
    &\left(\sharp W_{n(k,1)}^{N_{k}C_{k,1}}\cdots\sharp W_{n(k,s_{k})}^{N_{k}C_{k,s_{k}}}\right)^{i-1}\\
\geq&\exp\left(\sum_{j=1}^{s_{k}}(i-1)N_{k}C_{k+1,j}n(k,j)(1-\gamma)(h^{Kat}_{m_{k,j}}(f,\epsilon)-4\gamma)\right)\\
\geq&\exp\left((i-1)N_{k}(1-\gamma)(h^{Kat}_{\nu}(f,\epsilon)-4\gamma)\right).
\end{align*}
Thus we obtain
\begin{align*}
\#\mathcal{W}_{k,i}&\geq\exp\left\{\left(\sum_{l=1}^{k-1}T_lN_l+(i-1)N_{k}\right)(1-\gamma)(h^{Kat}_{\nu}(f,\epsilon)-4\gamma)\right\}.
\end{align*}
Since $\mathcal{C}'$
is a cover, each point of $\mathcal{W}_{k_0,i_0}$ has a prefix associated with some $B_{M_{k,i}}(x,\epsilon/2)\in\mathcal{C}'$.
We have
\begin{equation}
\begin{split}
    &\sum_{B_{M_{k,i}}(x,\epsilon/2)\in\mathcal{C}'}\exp\left\{-\left(\sum_{l=1}^{k-1}T_lN_l+(i-1)N_{k}\right)(1-\gamma)(h^{Kat}_{\nu}(f,\epsilon)-4\gamma)\right\}\\
\geq&\sum_{B_{M_{k,i}}(x,\epsilon/2)\in\mathcal{C}'}\frac{1}{\#\mathcal{W}_{k,i}}\geq 1.
\end{split}
\end{equation}
One can readily verify that
\begin{align*}
\lim_{k\to\infty}\frac{\sum_{l=1}^{k-1}T_lN_l+(i-1)N_{k}}{M_{k,i+1}}=1.
\end{align*}
We can take $k$ large enough such that
\begin{align*}
\frac{\sum_{l=1}^{k-1}T_lN_l+(i-1)N_{k}}{M_{k,i+1}}>1-\gamma.
\end{align*}
Thus when $k$ is large enough, we have
\begin{align*}
     &\sum\limits_{B_{M_{k,i}}(x,\epsilon/2)\in\mathcal{C}'}\exp\left(-M_{k,i+1}(1-\gamma)^2(h^{Kat}_{\nu}(f,\epsilon)-4\gamma)\right)\\
\geq &\sum\limits_{B_{M_{k,i}}(x,\epsilon/2)\in\mathcal{C}'}\exp\left\{-\left(\sum_{l=1}^{k-1}T_lN_l+(i-1)N_{k}\right)(1-\gamma)(h_\nu^{Kat}(f,\epsilon)-4\gamma)\right\}\\
\geq &1,
\end{align*}
which implies that
\begin{align*}
m\left(F,(1-\gamma)^2(h^{Kat}_{\nu}(f,\epsilon)-4\gamma),N,\epsilon/2\right)\geq1.
\end{align*}
for sufficiently large $N$. We deduce that
\begin{align*}
h_{top}(F,\frac{\epsilon}{2})\geq (1-\gamma)^2(h^{Kat}_{\nu}(f,\epsilon)-4\gamma)\geq(1-\gamma)^2(\mathbf{C}-5\gamma).
\end{align*}
Finally, by letting $\epsilon\to0$ and $\gamma\to0$, we obtain
\begin{align*}
h^B_{top}(F)\geq \mathbf{C},
\end{align*}
which completes the proof of theorem \ref{thm2.1}.

\section{Proof of Theorem \ref{thm2.2} and Theorem \ref{thm2.3}}
In this section, we prove the remaining theorem \ref{thm2.2} and theorem \ref{thm2.3}.
Firstly, we prove theorem \ref{thm2.2}.

By lemma \ref{lem4.1} and theorem \ref{thm2.2}, it suffices to show that $\sup\limits_{\nu\in K}h_\nu(f)\leq\sup\limits_{\nu\in intK}h_\nu(f)$.
Let $\mathbf{C}=\sup\limits_{\nu\in K}h_\nu(f)$. Fix $\epsilon>0$, choose $\nu\in K$
such that $h_\nu(f)>\mathbf{C}-\epsilon$. Let $\mu\in intK$. For $0<t<1$, we define $\gamma(t)=t\nu+(1-t)\mu$.
Then $\gamma(t)\in intK$. In fact, let $V$ be an open neighborhood of $\mu$ such that
$\mu\in V\subset intK$. Then $\gamma(t)\in t\nu+(1-t)V\subset K$, which implies $\gamma(t)\in intK$.
We can take $t$ close to $1$ such that $h_{\gamma(t)}(f)=th_\nu(f)+(1-t)h_\mu(f)>\mathbf{C}-\epsilon$,
which completes the proof of theorem \ref{thm2.2}.

Next, we prove theorem \ref{thm2.3}. By theorem \ref{thm2.1}, it suffices to prove the upper bound.
Let $\{\delta_n\}_{n\in\mathbb{N}}$ be a strictly decreasing sequence which tends to 0. We have
\begin{align*}
    \lim_{\delta\to0}h^B_{top}(\Delta_{cap}(B(\nu,\delta)),f)
=   \lim_{\delta\to0}\sup\limits_{\mu\in B(\nu,\delta)}h_\nu(f)
\leq\lim_{n\to\infty}\sup\limits_{\mu\in \overline{B}(\nu,\delta_n)}h_\nu(f).
\end{align*}
We will now prove that
\begin{align*}
\lim_{n\to\infty}\sup\limits_{\nu\in \overline{B}(\mu,\delta_n)}h_\nu(f)\leq h_\nu(f).
\end{align*}
Let $\epsilon>0$. For each $n$, we can choose $\nu_n\in \overline{B}(\nu,\delta_n)$
such that $h_{\nu_n}(f)>\sup\limits_{\mu\in \overline{B}(\nu,\delta_n)}h_\nu(f)-\epsilon$.
Since the entropy is upper semi-continuous, we obtain
\begin{align*}
  \lim_{n\to\infty}\sup\limits_{\mu\in \overline{B}(\nu,\delta_n)}h_\nu(f)
\leq \limsup_{n\to\infty} h_{\nu_{n}}(f)+\epsilon
\leq h_\nu(f)+\epsilon.
\end{align*}
Letting $\epsilon\to0$, the desired result follows.

\section{Proof of Theorem \ref{thm2.4}}
In this section, we verify theorem \ref{thm2.4}.
Up to minor modifications, the proof is identical with the proof of theorem \ref{thm2.1} and theorem \ref{thm2.2}.
By lemma \ref{lem4.1}, it suffices to obtain the lower bound estimate.
For reader's convenience, we sketch the proof of the
first result of theorem \ref{thm2.4}.

Let $\mathbf{C}:=\sup\left\{h_\nu(f):\nu\in U\right\}$.
Fix small $0<\delta,\gamma<1$ and $\gamma<\mathbf{C}/5$.
Choose a $\nu\in U$ such that
\begin{align*}
h_{\nu}(f)>\mathbf{C}-\gamma/2.
\end{align*}
By (\ref{equ3.1}), we can choose $\epsilon>0$ sufficiently small so that
\begin{align*}
\overline{B}(\nu,\frac{\epsilon}{2})\subset U,
h^{Kat}_{\nu}(f,\epsilon)>\mathbf{C}-\gamma.
\end{align*}

\begin{lem}
For any integer $k\geq1$,
there exists a finite convex combination of ergodic probability measures with
rational coefficients $\nu_k=\sum\limits_{j=1}^{s_k}a_{k,j}m_{k,j}$ such that
\begin{align*}
D(\nu,\nu_k)\leq\frac{1}{k},m_{k,j}({\rm supp}(\mu))=1, ~{\rm and} ~
h_{\nu}^{Kat}(f,\epsilon)\leq\sum_{j=1}^{s_{k}}a_{k,j}h_{m_{k,j}}^{Kat}(f,\epsilon).
\end{align*}
\end{lem}

For $\epsilon>0$, we choose a positive real number $\delta'$
such that every $\delta'$-pseudo-orbit can be
$\frac{\epsilon}{4}$-shadowed by an actual orbit.
Let $\xi$ be a finite partition of $X$ with
diam$(\xi)<\frac{\delta'}{3}$.
For $n\in\mathbb{N},$ we consider the set
\begin{align*}
\Lambda^n(m_{k,j})=\{x\in {\rm supp}(\mu):f^q(x)\in\xi(x){\rm ~for~some~} q\in[n,(1+\gamma)n]\cap\mathbb{N},\\
{\rm~and~}D(\mathscr{E}_m(x),m_{k,j})<\frac{1}{k} {\rm~for~all~} m\geq n\},
\end{align*}
where $\xi(x)$ is the element in $\xi$ containing $x.$

\begin{lem}
\begin{align*}
\lim_{n\to\infty}m_{k,j}(\Lambda^n(m_{k,j}))=m_{k,j}({\rm supp}(\mu))=1.
\end{align*}
\end{lem}

For each $m_{k,j}$, following the proof of theorem \ref{thm2.1},
we can obtain an integer $n(k,j)$ and an $(n(k,j),\epsilon)$-separated set $W_{n(k,j)}\subset {\rm supp}(\mu)$ satisfying
\begin{enumerate}
  \item[(1)] $W_{n(k,j)}\subset A_{n(k,j)}$, where $A_{n(k,j)}\in\xi$;
  \item[(2)] For $x\in W_{n(k,j)}$, $f^{n(k,j)}(x)\in A_{n(k,j)}$ and $D(\mathscr{E}_m(x),m_{k,j})<\frac{1}{k}$ for $m\geq n(k,j)$;
  \item[(3)] $\#W_{n(k,j)}\geq\exp\left(n(k,j)(1-\gamma)(h^{Kat}_{m_{k,j}}(f,\epsilon)-4\gamma)\right)$.
\end{enumerate}

Notice that $A_{n(k,j)}$ is contained in an open subset $U(k,j)$
with diam$(U(k,j))\leq3\text{diam}(\xi)<\delta'$. Obviously, $m_{k,j}(U(k,j))>0$.
By the ergodicity of $\mu$, for any two
measures $m_{k_1,j_1},m_{k_2,j_2}$,
there exist $s=s(k_1,j_1,k_2,j_2)\in\mathbb{N}$
and $y=y(k_1,j_1,k_2,j_2)\in U(k_1,j_1)$  such that
$f^s(y)\in U(k_2,j_2)$. The remaining proof is same the
proof of theorem \ref{thm2.1}.

\section{Application}
At last, we apply the above results to the study of multifractal structure of Birkhoff averages.
For a sequence $\{x_n\}_{n\in\mathbb{N}}$ in a metric space, let $A(x_n)$
denote the set of limit points of the sequence $\{x_n\}_{n\in\mathbb{N}}$.
Let $(X,f)$ be a TDS and $\varphi:X\to\mathbb{R}$ be a continuous function.
For $M\subset\mathbb{R}$, define
\begin{align*}
&\Delta'_{sub}(M):=\left\{x\in X:A\left(\frac{1}{n}S_n\varphi(x)\right)\subset M\right\},\\
&\Delta'_{cap}(M):=\left\{x\in X:A\left(\frac{1}{n}S_n\varphi(x)\right)\cap M\neq\emptyset \right\}.
\end{align*}

\begin{thm}
Let $(X,f)$ be a transitive TDS satisfying the
pseudo-orbit tracing property and $\varphi:X\to\mathbb{R}$ be a continuous function.
We have the following results.
\begin{enumerate}
\item If $U$ is an open subset of $\mathbb{R}$ and $U\cap\left\{\mu\in \mathscr M_{\rm inv}(X,f):\int\varphi d\mu\right\}\neq\emptyset$, then
 \begin{align*}
   h^B_{top}(\Delta'_{sub}(U),f)=h^B_{top}(\Delta'_{cap}(U),f)
   =\sup_{\nu\in \mathscr M_{\rm inv}(X,f)}\left\{h_\nu(f):\int\varphi d\nu\in U\right\}.
 \end{align*}
\item If $K\subset\mathbb{R}$ is a convex subset and $intK\cap\left\{\mu\in \mathscr M_{\rm inv}(X,f):\int\varphi d\mu\right\}\neq\emptyset$, then
 \begin{align*}
   h^B_{top}(\Delta'_{sub}(K),f)=h^B_{top}(\Delta'_{cap}(K),f)
   =\sup_{\nu\in \mathscr M_{\rm inv}(X,f)}\left\{h_\nu(f):\int\varphi d\nu\in K\right\}.
 \end{align*}
\end{enumerate}
\end{thm}
\begin{proof}
Let $F:\mathscr M_{\rm inv}(X,f)\to\mathbb{R}$ be a continuous function
defined by $\mu\mapsto\int\varphi d\mu$. The upper bound can be easily obtained by lemma \ref{lem4.1}.
It suffices to show
  \begin{align*}
   \Delta'_{sub}(U)=\Delta_{sub}(F^{-1}(U)).
 \end{align*}
In fact, if $x\in\Delta'_{sub}(U)$, then $A(\int\varphi d\mathscr E_n(x))\subset U$.
For $\mu\in V(x)$, we have $\int\varphi d\mu\in U$, which implies $x\in \Delta_{sub}(F^{-1}(U))$.
Conversely, if $x\in \Delta_{sub}(F^{-1}(U))$, then for any $\mu\in V(x)$, $\int\varphi d\mu\in U$,
which implies $A\left(\frac{1}{n}S_n\varphi(x)\right)\subset U$. Thus $x\in\Delta'_{sub}(U)$.
It follows that
 \begin{align*}
     &h^B_{top}(\Delta'_{sub}(U),f)\\
   =&h^B_{top}(\Delta_{sub}(F^{-1}(U)),f)\\
   =&\sup_{\nu\in F^{-1}(U)}h_\nu(f)\\
   =&\sup_{\nu\in \mathscr M_{\rm inv}(X,f)}\left\{h_\nu(f):\int\varphi d\nu\in U\right\}.
 \end{align*}
Similarly, one can prove the second result.
\end{proof}
If $(X,f)$ is not topologically transitive, we can give the following theorem. The proof is similar
to the proof of the above theorem.

\begin{thm}
Let $(X,f)$ be a transitive TDS satisfying the
pseudo-orbit tracing property and $\mu$ be an ergodic measure.
Let $\varphi\in C(X)$ and $F:\mathscr M_{\rm inv}(X,f)\to\mathbb{R}$ be a continuous function
defined by $\mu\mapsto\int\varphi d\mu$.
We have the following results.
\begin{enumerate}
\item If $U$ is an open subset of $\mathbb{R}$ and $\emptyset\neq F^{-1}(U)\subset \mathscr M_{\rm inv}({\rm supp}(\mu),f)$, then
 \begin{align*}
   h^B_{top}(\Delta'_{sub}(U),f)=h^B_{top}(\Delta'_{cap}(U),f)
   =\sup_{\nu\in \mathscr M_{\rm inv}(X,f)}\left\{h_\nu(f):\int\varphi d\nu\in U\right\}.
 \end{align*}
\item If $K\subset\mathbb{R}$ is a convex subset, $F^{-1}(K)\subset \mathscr M_{\rm inv}({\rm supp}(\mu),f)$ and $intF^{-1}(K)\neq\emptyset$, then
 \begin{align*}
   h^B_{top}(\Delta'_{sub}(K),f)=h^B_{top}(\Delta'_{cap}(K),f)
   =\sup_{\nu\in \mathscr M_{\rm inv}(X,f)}\left\{h_\nu(f):\int\varphi d\nu\in K\right\}.
 \end{align*}
\end{enumerate}
\end{thm}


\noindent {\bf Acknowledgements.}   The research was supported by
the National Basic Research Program of China
(Grant No. 2013CB834100) and the National Natural Science
Foundation of China (Grant No. 11271191).

\end{document}